\documentclass[12pt, reqno]{amsart}

\usepackage{amssymb, amsthm}
\usepackage{fullpage}
\usepackage[pdftex]{hyperref}

\newtheorem{theorem}{Theorem}
\newtheorem*{oldtheorem}{Theorem}
\newtheorem{lemma}{Lemma}[section]
\newtheorem{proposition}{Proposition}
\newtheorem{corollary}{Corollary}

\theoremstyle{remark}

\newtheorem*{notation}{Notation}
\newtheorem*{acknowledgement}{Acknowledgements}

\newcommand{\modulo}[1]{\mathrm{mod}\;#1}
\newcommand{\pmodulo}[1]{\;(\mathrm{mod}\;#1)}
\newcommand{\eps}{\varepsilon}
\newcommand{\sfrac}[2]{{\textstyle \frac{#1}{#2}}}

\title{Sums of Primes and Squares of Primes in Short Intervals}

\author{A.V. Kumchev}
\address{Department of Mathematics\\ Towson University\\ Towson, MD 21252\\ U.S.A.}
\email{akumchev@towson.edu}

\author{J.Y. Liu}
\address{Department of Mathematics\\ Shandong University\\ Jinan, Shandong 250100\\ P.R. China}
\email{jyliu@sdu.edu.cn}

\begin{document}

\begin{abstract}
  Let $\mathcal H_2$ denote the set of even integers $n \not\equiv 1 \pmodulo 3$. We prove that when $H \ge X^{0.33}$, almost all integers $n \in \mathcal H_2 \cap (X, X + H]$ can be represented as the sum of a prime and the square of a prime. We also prove a similar result for sums of three squares of primes. 
\end{abstract}

\subjclass[2000]{11P32, 11L20, 11N36.}

\maketitle

\section{Introduction}
\label{s0}

Additive prime number theory was ushered in by two seminal papers: I.M. Vinogradov's celebrated proof of the three primes theorem \cite{IVin37a} and L.K. Hua's work \cite{Hua38}. In the latter, Hua posed several questions that have represented the central problems in the field ever since. This note is concerned with two of those questions. Let
\begin{align*}
  \mathcal H_2 &= \big\{ n \in \mathbb N \; \big| \; n \not\equiv 1 \pmodulo {3}, \; 2 \mid n \big\}, \\
  \mathcal H_3 &= \big\{ n \in \mathbb N \; \big| \; n \equiv 3 \pmodulo {24}, \; 5 \nmid n \big\}.
\end{align*}
It is conjectured that every sufficiently large $n \in \mathcal H_2$ can be represented as the sum of a prime and the square of another prime, and that every sufficiently large integer $n \in \mathcal H_3$ can be represented as the sum of three squares of primes. However, both these conjectures are still wide open. Let $E_j(X)$ denote the number of integers $n \in \mathcal H_j$, with $n \le X$, which cannot be represented in the desired form. Hua \cite{Hua38} proved that
\begin{equation}\label{0.2}
  E_j(X) \ll X(\log X)^{-A} \qquad (j = 2, 3),
\end{equation}
for some $A > 0$. Later, Schwarz \cite{Schw61} showed that \eqref{0.2} holds for any fixed $A > 0$. Bauer~\cite{Baue99} and Leung and Liu \cite{LeLi93} used the method of Montgomery and Vaughan \cite{MoVa75} to prove that $E_j(X) \ll X^{1 - \delta_j}$ for some (very small) absolute constants $\delta_j > 0$. In the case of sums of three squares, there have been also a series of recent advances \cite{BaLiZh00, HaKu06, Kumc06a, LiZh98, LiZh01, LiZh05}, culminating in the result of Harman and the first author \cite{HaKu06} that $E_3(X) \ll X^{6/7 + \eps}$ for any fixed $\eps > 0$.

Zhan and the second author \cite{LiZh97} considered short interval versions of \eqref{0.2}. They obtained the following result.

\begin{oldtheorem}\label{th0}
  Let $A > 0$ and $\eps > 0$ be fixed. If $X^{7/16 + \eps} \le H \le X$, then
  \begin{equation}\label{0.3}
    E_2(X + H) - E_2(X) \ll H(\log X)^{-A}.
  \end{equation}
  Also, if $X^{3/4 + \eps} \le H \le X$, then
  \begin{equation}\label{0.4}
    E_3(X + H) - E_3(X) \ll H(\log X)^{-A}.
  \end{equation}
  The implied constants in \eqref{0.3} and \eqref{0.4} depend at most on $A$ and $\eps$.
\end{oldtheorem}

The admissible range for $H$ in the second part of this theorem was extended to $H \ge X^{1/2 + \eps}$ by Mikawa~\cite{Mika97}, and then recently to $H \ge X^{7/16 + \eps}$ by Mikawa and Peneva \cite{MiPe07}.

The proofs in \cite{LiZh97, Mika97, MiPe07} use the Hardy--Littlewood circle method to count representations of the desired form on average over $n$. For example, let $Y = X^{7/12 + \eps/2}$ and write
\begin{equation}\label{0.5}
  R_2(n) = \sum_{ \substack{ p_1 + p_2^2 = n\\ p_j \in \mathbf I_j}} 1,
\end{equation}
where $p_1$ and $p_2$ denote primes and
\begin{equation}\label{0.6}
  \mathbf I_1 = [X - Y, X), \quad \mathbf I_2 = \big[ \sfrac 12Y^{1/2}, Y^{1/2} \big).
\end{equation}
Deferring some standard notation to the end of this Introduction, we now define
\begin{align}
  r_2(n) &= r_2(n; X, Y) = \sum_{ \substack{ m_1 + m_2^2 = n\\ m_j \in \mathbf I_j}} \frac 1{(\log m_1)(\log m_2)}, \label{0.7}\\
  \mathfrak S_2(n, P) &= \sum_{q \le P} \frac {\mu(q)}{\phi(q)^2} \sum_{ \substack{ 1 \le a \le q\\ (a, q) = 1}} S(q, a)e(-an/q), \label{0.8}
\end{align}
where $m_1$ and $m_2$ denote integers and
\begin{equation}\label{0.1}
  S(q, a) = \sum_{ \substack{ 1 \le x \le q\\ (x, q) = 1}} e \big( ax^2/ q \big).
\end{equation}
The estimate \eqref{0.3} was established in \cite{LiZh97} by showing that when $H \ge Y^{3/4 + \eps/2}$, the asymptotic formula
\[
  R_2(n) = r_2(n)\mathfrak S_2(n, P) \big( 1 + O \big( (\log X)^{-A} \big) \big)
\]
holds for all but $O \big( H(\log X)^{-A} \big)$ integers $n \in \mathcal H_2 \cap (X, X + H]$. Here, $P = (\log X)^B$ for some $B = B(A) > 0$. In the present paper, we demonstrate how a rather simple sieve idea yields a similar result for $H \ge Y^{2/3 + \eps/2}$. This leads to the following theorem.

\begin{theorem}\label{thm1}
  Let $A > 0$, $\delta > 0$ and $\eps > 0$ be fixed, and suppose that $X^{7/18 + \eps} \le H \le X$. There exists a $B = B(A) > 0$ such that when $P = (\log X)^B$, the asymptotic formula
  \begin{equation}\label{0.9}
    R_2(n) = r_2(n)\mathfrak S_2(n, P) \big( 1 + O \big( (\log X)^{-1 + \delta} \big) \big)
  \end{equation}
  holds for all but $O \big( H(\log X)^{-A} \big)$ integers $n \in \mathcal H_2 \cap (X, X + H]$. The implied constants depend at most on $A, \delta$ and $\eps$.
\end{theorem}

In particular, it follows from this theorem that \eqref{0.3} holds when $H \ge X^{7/18 + \eps}$. The error term in \eqref{0.9} is somewhat weaker than the error term in the analogous result in \cite{LiZh97}, but that is a small price to pay for the longer range for $H$.

It appears very difficult to improve further on Theorem \ref{thm1}, if an asymptotic formula similar to \eqref{0.9} is required. On the other hand, if one is content merely with the existence of representations of $n$ as the sum of a prime and a square of a prime, then further progress is possible. Indeed, combining the circle method with Harman's sieve method (see \cite{Harm83a, Harm96}), we obtain the following result.

\begin{theorem}\label{thm2}
  Let $A > 0$ be fixed and suppose that $X^{0.33} \le H \le X$. Then \eqref{0.3} holds.
\end{theorem}

The exponent $0.33$ is not the exact limit of the method but just a reasonably close upper bound for that limit. It can be easily ``improved'' to $0.3275$ by choosing $\theta_2 = 0.595$ in the calculations in \S \ref{s5}. However, it appears that in order to replace $0.33$ by even $0.325$, one needs a substantially new idea.

The methods used in the proofs of Theorems \ref{thm1} and \ref{thm2} can be easily adapted to improve on the result of Mikawa and Peneva on sums of three squares of primes. In particular, when $X^{7/18 + \eps} \le H \le X$, we obtain an asymptotic result similar to Theorem \ref{thm1}. The application of the sieve method to this problem, on the other hand, is somewhat less successful. We obtain the following analogue of Theorem~\ref{thm2}.

\begin{theorem}\label{thm3}
  Let $A > 0$ be fixed and suppose that $X^{7/20} \le H \le X$. Then \eqref{0.4} holds.
\end{theorem}

One can use Theorem \ref{thm3} to estimate the number of exceptions in a short interval for representations as sums of four squares of primes.
Let $E_4(X)$ denote the number of integers $n$, with $n \le X$ and $n \equiv 4 \pmodulo {24}$, which cannot be represented as the sum of squares of primes. Combining Theorem \ref{thm3} with known results on the difference between two consecutive primes, we obtain the following result.
   
\begin{corollary}\label{c1}
  Let $A > 0$ be fixed and suppose that $X^{0.27} \le H \le X$. Then
  \[
    E_4(X + H) - E_4(X) \ll H(\log X)^{-A}.
  \]
\end{corollary}

\begin{notation}
  Throughout the paper, the letter $p$, with or without indices, is reserved for prime numbers; $c$ denotes an absolute constant, not necessarily the same in all occurrences. As usual in number theory, $\mu(n)$, $\phi(n)$ and $\tau(n)$ denote, respectively, the M\"obius function, Euler's totient function and the number of divisors function; $\| x \|$ denotes the distance from $x$ to the nearest integer. We write $e(x) = \exp( 2\pi ix )$, $e_q(x) = e(x/q)$, and $(a, b) = \mathrm{gcd}(a, b)$. Also, we use $m \sim M$ and $m \asymp M$ as abbreviations for the conditions $M \le m < 2M$ and $c_1M \le m < c_2M$. Finally, if $z \ge 2$, we define $\Pi(z) = \prod_{p \le z} p$ and introduce the functions
  \begin{align}
    \Phi(n, z) &= \begin{cases}
      1 & \text{if } p \mid n \implies p \ge z, \\
      0 & \text{otherwise};
    \end{cases}  \label{0.10} \\
    \Psi(n, z) &= \begin{cases}
      1 & \text{if } p \mid n \implies p \le z, \\
      0 & \text{otherwise}.
    \end{cases}  \label{0.11}
  \end{align}
\end{notation}

\section{Outline of the method}
\label{s1}

In this section, we outline the proofs of Theorems \ref{thm1} and \ref{thm2}. The details of those proofs are presented in~\S\ref{s4} and \S\ref{s5}. The proof of Theorem \ref{thm3} and its corollary are given in \S\ref{s6}.
 
\subsection{The circle method}
\label{s1.1}

Suppose that $X$ is a large real, and let $L = \log X$, $Y = X^{\theta_1}$, $H = Y^{\theta_2}$, where $\theta_1$ and $\theta_2$ are positive constants to be specified later. Also, let $\mathbf I_1$ and $\mathbf I_2$ be the intervals \eqref{0.6} with $Y = X^{\theta_1}$. For any pair of arithmetic functions $\lambda_1, \lambda_2$, put
\begin{equation}\label{e2.1}
  R(n; \lambda_1, \lambda_2) = \sum_{ \substack{ m_1 + m_2^2 = n\\ m_i \in \mathbf I_i}} \!\!\! \lambda_1(m_1)\lambda_2(m_2).
\end{equation}
In particular, we have $R_2(n) = R(n; \varpi, \varpi)$, where $\varpi$ is the characteristic function of the primes. In the proofs of Theorems \ref{thm1} and \ref{thm2}, we apply the circle method to $R(n; \lambda_1, \lambda_2)$ with different choices of $\lambda_1$ and $\lambda_2$.

The application of the circle method starts with the identity
\begin{equation}\label{e2.2}
  R(n; \lambda_1, \lambda_2) = \int_0^1 S_1(\alpha)S_2(\alpha)e(-\alpha n) \, d\alpha,
\end{equation}
where
\[
  S_j(\alpha) = \sum_{m \in \mathbf I_j} \lambda_j(m)e\big( \alpha m^j \big) \qquad (j = 1, 2).
\]
Suppose that $A > 0$ is a fixed real, which we assume to be larger than some absolute constant. We set
\begin{equation}\label{e2.3}
  P = L^B, \quad Q_0 = YP^{-3}, \quad Q = HP^{-1},
\end{equation}
where $B$ is a parameter to be chosen later in terms of $A$. We define the sets of major and minor arcs as follows:
\begin{equation}\label{e2.4}
  \mathfrak M = \bigcup_{q \le P} \bigcup_{ \substack{ 1 \le a \le q\\ (a, q) = 1}} \bigg[ \frac aq - \frac 1{qQ}, \frac aq + \frac 1{qQ} \bigg], \quad \mathfrak m = \big[ Q^{-1}, 1 + Q^{-1} \big] \setminus \mathfrak M.
\end{equation}
We also write $\mathfrak M(q, a) = \big\{ \alpha \in \mathbb R \; \big| \; |q\alpha - a| < Q^{-1} \big\}$.

In order to proceed further, we need to make some assumptions regarding $\lambda_1$ and $\lambda_2$. We require the following hypotheses:
\begin{enumerate}
  \item [(A$_{j.1}$)] We have $\lambda_j(m) \ll 1$ and $\lambda_j(m) = 0$ when $\Phi(m, P) = 0$. 
  \item [(A$_{1.2}$)] There exists a smooth function $f_1$ such that the inequality
    \[
      \sup_{\mathbf J \subseteq \mathbf I_1} \bigg| \sum_{m \in \mathbf J} \big( \lambda_1(m) - D(\chi)f_1(m) \big)\chi(m) \bigg| \ll YP^{-5}
    \]
    holds for all Dirichlet characters $\chi$ with moduli $q \le P$. Here, the supremum is over all subintervals of\/ $\mathbf I_1$, and $D(\chi) = 1$ or $0$ according as $\chi$ is principal or not. 
  \item [(A$_{2.2}$)] There exists a smooth function $f_2$ such that the inequality
    \[
      \int_{Y/4}^Y \bigg| \sum_{t < m^2 \le t + \delta t} \big( \lambda_2(m) - D(\chi)f_2(m) \big)\chi(m) \bigg|^2 \, dt \ll (qQ)^2 P^{-4}
    \]
    holds for all Dirichlet characters $\chi$ with moduli $q \le P$ and all real $\delta$ with $0 < \delta \ll qQY^{-1}$.
\end{enumerate}

When $\alpha \in \mathfrak M(q, a)$, we define the functions
\[
  S_1^*(\alpha) = \frac {\mu(q)}{\phi(q)} T_1(\alpha - a/q), \quad S_2^*(\alpha) = \frac {S(q, a)}{\phi(q)} T_2(\alpha - a/q),
\]
where $S(q, a)$ is defined in \eqref{0.1} and
\[
  T_j(\beta) = \sum_{m \in \mathbf I_j} f_j(m) e \big( \beta m^j \big) \qquad (j = 1, 2).
\]
Since the intervals $\mathfrak M(q, a)$ are disjoint, this defines $S_j^*(\alpha)$ on $\mathfrak M$. The analysis of the major arcs aims to prove that one can approximate $S_j(\alpha)$ by $S_j^*(\alpha)$ on average over $\alpha \in \mathfrak M$. By Cauchy's inequality,
\begin{equation}\label{e2.5}
  \int_{\mathfrak M} |S_1(S_2 - S_2^*)| \, d\alpha \le PI_1^{1/2} \bigg( \max_{ \substack{ 1 \le a \le q \le P\\ (a, q) = 1}} \int_{\mathfrak M(q, a)} |S_2 - S_2^*|^2 \, d\alpha \bigg)^{1/2},
\end{equation}
where
\begin{equation}\label{e2.6}
  I_1 = \int_0^1 |S_1|^2 \, d\alpha = \sum_{m \in \mathbf I_1} \lambda_1(m)^2 \ll Y.
\end{equation}
Let $\alpha \in \mathfrak M(q, a)$ and note that (A$_{2.1}$) implies that $\lambda_2(m) = 0$ when $(m, q) > 1$. Using the orthogonality of the characters modulo $q$, we obtain
\begin{equation}\label{e2.7}
  | S_2(\alpha) - S_2^*(\alpha)|^2 \le \sum_{\chi \, \modulo q} |W_2(\alpha - a/q, \chi)|^2,
\end{equation}
where
\[
  W_j(\beta, \chi) = \sum_{m \in \mathbf I_j} (\lambda_j(m) - D(\chi)f_j(m)) \chi(m) e\big( \beta m^j \big) \qquad (j = 1, 2).
\]
Inserting \eqref{e2.6} and \eqref{e2.7} into the right side of \eqref{e2.5}, we get
\begin{equation}\label{e2.8}
  \int_{\mathfrak M} |S_1(S_2 - S_2^*)| \, d\alpha \ll PY^{1/2} \bigg( \max_{q \le P} \sum_{\chi \, \modulo q} \int_{-1/(qQ)}^{1/(qQ)} |W_2(\beta, \chi)|^2 \, d\beta \bigg)^{1/2}.
\end{equation}
Combining Gallagher's lemma \cite[Lemma 1]{Gall70}) with a device of Saffari and Vaughan \cite[p. 25]{SaVa77}, we find that
\[
  \int_{-1/(qQ)}^{1/(qQ)} |W_2(\beta, \chi)|^2 \, d\beta \ll \frac 1{(qQ)^2} \int_{Y/4}^Y \bigg| \sum_{t < m^2 \le t + \delta t} \big( \lambda_2(m) - D(\chi)f_2(m) \big)\chi(m) \bigg|^2 dt + \delta,
\]
for some $\delta \asymp (qQ)Y^{-1} \ll HY^{-1}$. Thus, by \eqref{e2.3}, \eqref{e2.8} and hypothesis (A$_{2.2}$) above,
\begin{equation}\label{e2.9}
  \int_{\mathfrak M} |S_1(S_2 - S_2^*)| \, d\alpha \ll (Y/P)^{1/2}.
\end{equation}

Before proceeding further, we make an assumption regarding the smooth functions $f_1$ and $f_2$ appearing in hypotheses (A$_{j.2}$): we suppose that
\begin{equation}\label{e2.10}
  |f_j(m)| \ll 1, \quad |f_j'(m)| \ll (1 + |m|)^{-1} \qquad (j = 1, 2).
\end{equation}
These simple conditions suffice to deduce the bounds
\begin{equation}\label{e2.11}
  T_j(\beta) \ll Y^{1/j}(1 + Y|\beta|)^{-1} \qquad (j = 1, 2).
\end{equation}
Let
\begin{equation}\label{e2.11a}
  \mathfrak M_0 = \bigcup_{q \le P} \bigcup_{ \substack{ 1 \le a \le q\\ (a, q) = 1}} \bigg[ \frac aq - \frac 1{qQ_0}, \frac aq + \frac 1{qQ_0} \bigg], \quad \mathfrak m_0 = \mathfrak M \setminus \mathfrak M_0.
\end{equation}
By \eqref{e2.3}, \eqref{e2.6} and \eqref{e2.11},
\begin{align}\label{e2.12}
  \bigg( \int_{\mathfrak m_0} |S_1S_2^*| \, d\alpha \bigg)^2 &\ll I_1 \sum_{q \le P} \sum_{1 \le a \le q} \frac {|S(q, a)|^2}{\phi(q)^2} \int_{1/(qQ_0)}^{1/2} |T_2(\beta)|^2 \, d\beta \notag\\
  &\ll Y^2 \sum_{q \le P} q^{\eta} \int_{1/(qQ_0)}^{\infty} \frac {d\beta}{(1 + Y|\beta|)^2} \ll YP^{-1 + \eta}. 
\end{align}
Here and through the remainder of this section, $\eta > 0$ is a fixed real that can be taken arbitrarily small.

Now, if $\alpha \in \mathfrak M(q, a) \cap \mathfrak M_0$, we have (similarly to \eqref{e2.7})
\begin{equation}\label{e2.13}
  | S_1(\alpha) - S_1^*(\alpha)| \ll q^{-1/2+\eta} \sum_{\chi \, \modulo q} |W_1(\alpha - a/q, \chi)|.
\end{equation}
Using partial summation, we deduce from \eqref{e2.13} and (A$_{1.2}$) that
\[
  | S_1(\alpha) - S_1^*(\alpha)| \ll q^{1/2+\eta/2} YP^{-10}(1 + Y|\beta|).
\]
From this inequality and \eqref{e2.11}, we obtain
\begin{align}\label{e2.14}
  \int_{\mathfrak M_0} |(S_1 - S_1^*)S_2^*| \, d\alpha &\ll \sum_{q \le P} q^{\eta} \sum_{1 \le a \le q} \int_{-1/(qQ_0)}^{1/(qQ_0)} Y^{3/2}P^{-5} \, d\beta\ll Y^{1/2}P^{-1 + \eta}.
\end{align}

Finally, by \eqref{e2.3} and \eqref{e2.11},
\[
  \int_{1/(qQ_0)}^{1/2} T_1(\beta)T_2(\beta) \, d\beta \ll Y^{1/2}P^{-2},
\]
whence
\begin{align}\label{e2.15}
  \int_{\mathfrak M_0} S_1^*(\alpha)S_2^*(\alpha)e(-\alpha n) \, d\alpha = \mathfrak S_2(n, P)\mathfrak I(n; \lambda_1, \lambda_2) + O \big(Y^{1/2}P^{-1} \big),
\end{align}
where $\mathfrak S_2(n, P)$ is defined in \eqref{0.8} and
\[
  \mathfrak I(n; \lambda_1, \lambda_2) = \int_{-1/2}^{1/2} T_1(\beta)T_2(\beta)e(-\beta n) \, d\beta = \sum_{ \substack{ m_1 + m_2^2 = n\\ m_i \in \mathbf I_i}} f_1(m_1)f_2(m_2).
\]
Combining \eqref{e2.9}, \eqref{e2.12}, \eqref{e2.14} and \eqref{e2.15}, we get
\begin{align}\label{e2.16}
\int_{\mathfrak M} S_1S_2e(-\alpha n) \, d\alpha
&= \int_{\mathfrak M_0} S_1^* S_2^*e(-\alpha n) \, d\alpha + \int_{\mathfrak M} S_1(S_2-S_2^*)e(-\alpha n) \, d\alpha \notag\\
&\quad +\int_{\mathfrak m_0} S_1 S_2^*e(-\alpha n) \, d\alpha + \int_{\mathfrak M_0} (S_1-S_1^*)S_2^*e(-\alpha n) \, d\alpha \notag\\
&= \mathfrak S_2(n, P)\mathfrak I(n; \lambda_1, \lambda_2) + O \big(Y^{1/2}P^{-1/2 + \eta} \big). 
\end{align}

In order to estimate the contribution from the minor arcs, we now make another hypothesis regarding $\lambda_2$:
\begin{enumerate}
  \item [(A$_{2.3}$)] Given any $A > 0$, there exists a $B_0 = B_0(A) > 0$ such that when $B \ge B_0$, the inequality
  \[
    \int_{Y/4}^Y \bigg| \sum_{t < m^2 \le t + H} \lambda_2(m) e \big( \alpha m^2 \big) \bigg|^2 dt \ll H^2L^{-A}
  \]
  holds for all $\alpha \in \mathfrak m$.
\end{enumerate}
Using the well-known bound
\[
  \sum_{X < n \le X + H} e(\alpha n) \ll \min \big( H, \| \alpha \|^{-1} \big),
\]
we obtain
\begin{align}\label{e2.17}
  &\phantom{\ll{}} \sum_{X < n \le X + H} \bigg| \int_{\mathfrak m} S_1(\alpha)S_2(\alpha)e(-\alpha n) \, d\alpha \bigg|^2 \notag\\
  &\ll \int_{\mathfrak m}\int_{\mathfrak m} |S_1(\alpha)S_2(\alpha)S_1(\beta)S_2(\beta)| \big( H, \| \alpha - \beta \|^{-1} \big) \, d\alpha d\beta \notag\\
  &\ll \int_{\mathfrak m}\int_{\mathfrak m} |S_1(\beta)S_2(\alpha)|^2 \big( H, \| \alpha - \beta \|^{-1} \big) \, d\alpha d\beta \notag\\
  &\ll I_1 \max_{\beta \in [0, 1]} \int_{\mathfrak m} |S_2(\alpha)|^2 \big( H, \| \alpha - \beta \|^{-1} \big) \, d\alpha. 
\end{align}
Moreover, a simple subdivision argument yields
\begin{align}\label{e2.18}
  \int_{\mathfrak m} |S_2(\alpha)|^2 \big( H, \| \alpha - \beta \|^{-1} \big) \, d\alpha \ll HL \int_{\mathbf J_{\gamma}} |S_2(\alpha)|^2 \, d\alpha,
\end{align}
for some $\gamma \in [0, 1]$ and $\mathbf J_{\gamma} = \mathfrak m \cap \big[ \gamma - H^{-1}, \gamma + H^{-1} \big]$. Since an interval of length $2H^{-1}$ can intersect at most one major arc, $\mathbf J_{\gamma}$ is either an interval or the union of two intervals. Hence,
\begin{align}\label{e2.19}
  \int_{\mathbf J_{\gamma}} |S_2(\alpha)|^2 \, d\alpha \ll \int_{-1/H}^{1/H} |S_2(\alpha + \beta)|^2 \, d\beta,
\end{align}
for some $\alpha \in \mathfrak m$. By Gallagher's lemma and hypothesis (A$_{2.3}$), the last integral is $O \big( L^{-A - 1} \big)$, which together with \eqref{e2.17}--\eqref{e2.19} gives
\begin{equation}\label{e2.20}
  \sum_{X < n \le X + H} \bigg| \int_{\mathfrak m} S_1(\alpha)S_2(\alpha)e(-\alpha n) \, d\alpha \bigg|^2 \ll HYL^{-A}.
\end{equation}

Combining \eqref{e2.2}, \eqref{e2.4}, \eqref{e2.16} and \eqref{e2.20}, we obtain the following result.

\begin{proposition}\label{p1}
  Let $A > 0$ be fixed and $P = L^B$, with $B \ge B_0(A) > 0$. Suppose that $\lambda_1$ is an arithmetic function satisfying hypotheses $(\mathrm A_{1.1})$ and  $(\mathrm A_{1.2})$, and that $\lambda_2$ is an arithmetic functions satisfying hypotheses $(\mathrm A_{2.1})$--$(\mathrm A_{2.3})$. Furthermore, suppose that the functions $f_1$ and $f_2$ appearing in hypotheses $(\mathrm A_{j.2})$ satisfy \eqref{e2.10}. Then
  \begin{equation}\label{e2.21}
    \sum_{X < n \le X + H} \big| R(n; \lambda_1, \lambda_2) - \mathfrak S_2(n, P)\mathfrak I(n; \lambda_1, \lambda_2) \big|^2 \ll HYL^{-A}.
  \end{equation}
\end{proposition}

\subsection{The sieve method}
\label{s1.2}

One can use Proposition \ref{p1} with $\lambda_1 = \lambda_2 = \varpi$, the characteristic function of the primes, to obtain an asymptotic formula for $R_2(n)$ for almost all $n \in \mathcal H_2 \cap (X, X + H]$ (that is, for all but $O\big( HL^{-A} \big)$ such $n$). However, when one tries to verify the hypotheses of the proposition, one is forced to choose $\theta_1 > \frac 7{12}$ and $\theta_2 > \frac 34$, and so one recovers the result of Zhan and the second author mentioned in the Introduction. Thus, in the proofs of the theorems, we use different choices for $\lambda_1$ and $\lambda_2$.

First, let
\begin{equation}\label{e2.22}
  \lambda_0(m) = \Phi(m, z_0), \quad z_0 = Y^{1/4}P^{-2}.
\end{equation}
We note that
\begin{equation}\label{e2.23}
  R_2(n) = R(n; \varpi, \lambda_0) - R_0(n),
\end{equation}
where $R_0(n)$ is the number of solutions of the equation
\begin{equation*}
  p_1 + (p_2p_3)^2 = n
\end{equation*}
in primes $p_1, p_2, p_3$ subject to
\begin{equation*}\label{1.5}
  p_1 \in \mathbf I_1, \quad z_0 < p_2 \le Y^{1/4}, \quad p_2 \le p_3, \quad p_2p_3 \in \mathbf I_2.
\end{equation*}
It turns out that Proposition \ref{p1} can be applied to $R(n; \varpi, \lambda_0)$ when $\theta_1 > \frac 7{12}$ and $\theta_2 > \frac 23$. This yields the asymptotic formula
\[
  R(n; \varpi, \lambda_0) = \mathfrak S_2(n, P)\mathfrak I(n) + O \big( Y^{1/2}L^{-A} \big)
\]
for almost all $n \in \mathcal H_2 \cap (X, X + H]$. Here, we have $\mathfrak I(n) \sim r_2(n)$, so this asymptotic formula is very close to the conjectured asymptotic formula for $R_2(n)$. In order to complete the proof of Theorem \ref{thm1}, we shall use an upper-bound sieve to show that
\begin{equation}\label{e2.24}
  R_0(n) \ll \mathfrak S_2(n, P)r_2(n)L^{-1+\delta}
\end{equation}
for almost all $n \in \mathcal H_2 \cap (X, X + H]$.

We now proceed to outline the proof of Theorem \ref{thm2}. We introduce two pairs of arithmetic functions: $\lambda_1^{\pm}$ such that
\begin{equation}\label{e2.25}
  \lambda_1^-(m) \le \varpi(m) \le \lambda_1^+(m) \qquad (m \in \mathbf I_1),
\end{equation}
and $\lambda_2^{\pm}$ such that
\begin{equation}\label{e2.26}
  \lambda_2^-(m) \le \lambda_0(m) \le \lambda_2^+(m) \qquad (m \in \mathbf I_2).
\end{equation}
Then
\begin{equation}\label{e2.27}
  R(n; \varpi, \lambda_0) \ge R(n; \lambda_1^+, \lambda_2^-) + R(n; \lambda_1^-, \lambda_2^+) - R(n; \lambda_1^+, \lambda_2^+).
\end{equation}
We remark that this inequality is a variant of the vector sieve of Br\"udern and Fouvry \cite{BrFo94}. We shall use Harman's sieve to construct the functions $\lambda_i^{\pm}$ so that Proposition \ref{p1} can be applied to each of the three terms on the right side of \eqref{e2.27}. It will then follow from \eqref{e2.27} that
\begin{equation}\label{e2.28}
  R(n; \varpi, \lambda_0) \ge ( \sigma(\theta_1, \theta_2) + o(1) ) r_2(n)\mathfrak S_2(n, P)
\end{equation}
for almost all $n \in \mathcal H_2 \cap (X, X + H]$. Here, $\sigma(\theta_1, \theta_2)$ is independent of any parameters other than the exponents
$\theta_1$ and $\theta_2$. Moreover, as a function of $\theta_1$ and $\theta_2$, $\sigma$ is continuous and non-decreasing with respect to each variable. Since $\sigma(0.55 + \eps, 0.6) \ge 0.17$, Theorem \ref{thm2} follows readily from \eqref{e2.23}, \eqref{e2.24} and \eqref{e2.28}.

\section{Lemmas}
\label{s2}

In this section, we collect various auxiliary results required in the proofs of the theorems. These lemmas fall in three major categories: bounds for exponential sums; results from elementary number theory and sieve theory; and results concerning the singular series.

\subsection{Bounds for exponential sums}
\label{s2.1}

The first two lemmas are essentially restatements of Lemmas~3.1 and~3.2 in \cite{LiZh97}. We omit the proofs, since they are identical to the proofs in \cite{LiZh97}.

\begin{lemma}\label{l1}
  Let $A > 0$, $B > 0$, and $x^{1/2} \le y \le x$, with $x$ sufficiently large. Suppose that $\alpha \in \mathbb R$ and $a, q \in \mathbb Z$ are such that
  \begin{equation}\label{2.1}
    L^B \le q \le yL^{-B}, \quad (a, q) = 1, \quad |\alpha - a/q| < q^{-2},
  \end{equation}
  where $L = \log x$. Suppose also that $(a_m)$ is a sequence of complex numbers with $|a_m| \le \tau(m)^c$, and that
  \begin{equation*}\label{2.2}
    1 \le M \le x^{1/4}L^{-B}.
  \end{equation*}
  Then, for $B \ge B_0(A) > 0$, one has
  \begin{equation*}\label{2.3}
    \int_x^{2x} \bigg| \mathop{ \sum_{m \sim M} \sum_k }_{t < m^2k^2 \le t + y} a_m e \big( \alpha m^2k^2 \big) \bigg|^2 dt \ll y^2L^{-A}.
  \end{equation*}
\end{lemma}

\begin{lemma}\label{l2}
  Let $A > 0$, $B > 0$, and $x^{1/2} \le y \le x$, with $x$ sufficiently large. Suppose that $\alpha \in \mathbb R$ and $a, q \in \mathbb Z$ satisfy \eqref{2.1}. Suppose also that $(a_m)$ and $(b_k)$ are sequences of complex numbers with $|a_m| \le \tau(m)^c$ and $|b_k| \le \tau(k)^c$, and that
  \begin{equation*}\label{2.4}
    L^B \le M \le yx^{-1/2}L^{-B}.
  \end{equation*}
  Then, for $B \ge B_0(A) > 0$, one has
  \begin{equation*}\label{2.5}
    \int_x^{2x} \bigg| \mathop{ \sum_{m \sim M} \sum_k }_{t < m^2k^2 \le t + y} a_mb_k e \big( \alpha m^2k^2 \big) \bigg|^2 dt \ll y^2L^{-A}.
  \end{equation*}
\end{lemma}

The next lemma is a simple tool for reducing the estimation of a bilinear sum to the estimation of a similar sum subject to `nicer' summation conditions. The proof is a standard application of Perron's integral formula, so we omit it and refer the reader to Kumchev \cite[Lemma 2.7]{Kumc06a}.

\begin{lemma}\label{lA}
  Let $F: \mathbb N \to \mathbb C$ satisfy $|F(x)| \le X$, let $M, K \ge 2$, and define the bilinear form
  \[
    \mathcal B(M, K) = \mathop{\sum_{m \sim M} \sum_{k \sim K}}_{m < k} a_m b_k F(mk),
  \]
  where $|a_m| \le 1$, $|b_k| \le 1$. Then
  \[
    \mathcal B(M, K) \ll L \bigg| \sum_{m \sim M} \sum_{k \sim K} a_m' b_k' F(mk) \bigg| + (XMK)^{-1},
  \]
  where $|a_m'| \le |a_m|$, $|b_k'| \le |b_k|$ and $L = \log(2MKX)$. The same estimate holds, if we replace the summation condition $m < k$ in the definition of $\mathcal B(M, K)$ with $U \le mk < U'$.
\end{lemma}

\begin{lemma}\label{l3}
  Let $A > 0$, $B > 0$, and $x^{1/2} \le y \le x$, with $x$ sufficiently large. Suppose that $\alpha \in \mathbb R$ and $a, q \in \mathbb Z$ satisfy \eqref{2.1}. Suppose also that $(a_m)$ is a sequence of complex numbers with $|a_m| \le \tau(m)^c$, and that
  \begin{equation*}\label{2.6}
    1 \le M \le x^{1/4}L^{-2B}, \quad 2 \le z \le yx^{-1/2}L^{-2B}.
  \end{equation*}
  Then, for $B \ge B_0(A) > 0$, one has
  \begin{equation*}\label{2.7}
    \int_x^{2x} \bigg| \mathop{ \sum_{m \sim M} \sum_k }_{t < m^2k^2 \le t + y} a_m \Phi(k, z) e \big( \alpha m^2k^2 \big) \bigg|^2 dt \ll y^2L^{-A},
  \end{equation*}
  where $\Phi(k, z)$ is the function defined in \eqref{0.10}.
\end{lemma}

\begin{proof}
  Let $g_t$ denote the indicator function of the interval $\big( t^{1/2}, (t + y)^{1/2} \big]$. We have
  \[
    \sum_{m \sim M} \sum_k a_m \Phi(k, z) g_t(mk) e \big( \alpha m^2k^2 \big) = \sum_{d \mid \Pi(z)} \sum_{m \sim M} \sum_k a_m \mu(d) g_t(mkd) e \big( \alpha m^2k^2d^2 \big).
  \]
  It thus suffices to show that
  \begin{equation}\label{2.8}
    \int_x^{2x} \bigg| \sum_{ \substack{ d \mid \Pi(z)\\ d \sim D}} \sum_{m \sim M} \sum_k a_m \mu(d) g_t(mkd) e \big( \alpha m^2k^2d^2 \big) \bigg|^2 dt \ll y^2L^{-A},
  \end{equation}
  where $1 \le D \ll x^{1/2}M^{-1}$. We distinguish three cases depending on the size of $D$.

  \paragraph*{\em Case 1:} $D \le L^B$. Upon defining the convolution
  \[
    b_r = \sum_{ \substack{ dm = r\\ d \sim D, m \sim M\\ d \mid \Pi(z)}} a_m \mu(d),
  \]
  we can rewrite the left side of \eqref{2.8} as
  \[
    \int_x^{2x} \bigg| \sum_{ r \sim R} \sum_k b_r g_t(rk) e \big( \alpha r^2k^2 \big) \bigg|^2 dt,
  \]
  where $|b_r| \le \tau(r)^c$ and $R = MD \le x^{1/4}L^{-B}$. Therefore, \eqref{2.8} follows from Lemma \ref{l1}.

  \paragraph*{\em Case 2:} $L^B \le D \le yx^{-1/2}L^{-B}$. Upon defining the convolution
  \[
    b_r = \sum_{ \substack{ mk = r\\ m \sim M}} a_m,
  \]
  we can rewrite the left side of \eqref{2.8} as
  \[
    \int_x^{2x} \bigg| \sum_{ d \sim D, d \mid \Pi(z)} \sum_r b_r\mu(d) g_t(rd) e \big( \alpha r^2d^2 \big) \bigg|^2 dt,
  \]
  where $|b_r| \le \tau(r)^c$. Therefore, \eqref{2.8} follows from Lemma \ref{l2}.

  \paragraph*{\em Case 3:} $D \ge yx^{-1/2}L^{-B}$. Set $V = yx^{-1/2}L^{-B}$. Each $d$ appearing in the summation has a factorization $d = p_1 \cdots p_r$ subject to
  \[
    p_r < \dots < p_1 < z, \quad p_1 \cdots p_r \ge V.
  \]
  Therefore, there is a unique integer $s$, $1 \le s < r$, such that
  \[
    L^B \le z^{-1}V \le p_1 \cdots p_s \le V \le p_1 \cdots p_{s + 1}.
  \]
  On writing $p = p_s$, $p' = p_{s + 1}$, $d_1 = p_1 \cdots p_{s - 1}$, $d_2 = p_{s + 2} \cdots p_r$, we can express the left side of \eqref{2.8} as
  \[
    \int_x^{2x} \bigg| \sum_{p, p'} \sum_{d_1, d_2} \sum_{m \sim M} \sum_k a_m \mu(d_1) \mu(d_2) \psi(d_1, p) g_t(mkpp'd_1d_2) e\big( \alpha (mkpp'd_1d_2)^2 \big) \bigg|^2 dt,
  \]
  where $p, p', d_1, d_2$ are subject to
  \[
    p' < p < z, \quad pp'd_1d_2 \sim D, \quad d_1 \mid \Pi(z), \quad d_2 \mid \Pi(p'), \quad L^B \le d_1p < V \le d_1pp'.
  \]
  Hence, using Lemma \ref{lA} to remove the summation conditions
  \[
    p' < p, \quad pp'd_1d_2 \sim D, \quad \text{and} \quad d_1pp' \ge V,
  \]
  we can show that the left side of \eqref{2.8} is bounded by
  \[
    L^c \int_x^{2x} \bigg| \sum_{L^B \le v \le V} \sum_u \tilde a_v b_u g_t(uv) e\big( \alpha u^2v^2 \big) \bigg|^2 dt + L^c,
  \]
  with coefficients $|\tilde a_v| \le 1$ and $|b_u| \le \tau(u)^c$ (the new variables being $u = mkp'd_2$ and $v = pd_1$). Thus, \eqref{2.8} follows from Lemma \ref{l2}.
\end{proof}

\subsection{Some lemmas from sieve theory}
\label{s2.2}

Let $\Phi(m, z)$ and $\Psi(m, z)$ be the functions defined in \eqref{0.10} and \eqref{0.11}. Lemma \ref{l4} below is Theorem 1 in Tenenbaum \cite[\S III.5]{Tene95}. Lemma \ref{l5} is a variant of Theorem 3 in Tenenbaum \cite[\S III.6]{Tene95}.

\begin{lemma}\label{l4}
  If $x$ and $z$ are large real numbers, then
  \[
    \sum_{m \le x} \Psi(m, z) \ll x\exp \big( - (\log x)/(2\log z) \big).
  \]
\end{lemma}

\begin{lemma}\label{l5}
  Let $2 \le z \le x \le z^c$, and let $w$ be the continuous solution of the differential delay equation
  \[
    \begin{cases}
      (tw(t))' = w(t - 1) & \text{if } t > 2, \\
      w(t) = t^{-1}       & \text{if } 1 < t \le 2.
    \end{cases}
  \]
  Then for any fixed $A > 0$,
  \[
    \sum_{m \le x} \Phi(m, z) = \frac 1{\log z} \sum_{z < m \le x} w \left( \frac {\log m}{\log z} \right) + O \left( x (\log x)^{-A} \right),
  \]
the implied constant depending at most on $A$.
\end{lemma}

We now introduce some standard sieve-theoretic notation. If $\mathcal A$ is an integer sequence, we define
\[
  \mathcal A_d = \big\{ a \in \mathcal A \; | \; m \equiv 0 \pmodulo d \big\}.
\]
Suppose that when $d$ is squarefree, we have
\begin{equation}\label{2.9}
  |\mathcal A_d| = g(d)N + r(d),
\end{equation}
where $N$ is a large parameter independent of $d$ and $g$ is a multiplicative function such that $0 \le g(p) < 1$ for all $p$. We assume that there exist constants $\kappa \ge 0$ and $K \ge 2$ such that
\begin{equation}\label{2.10}
  \prod_{w \le p < z} \big( 1 - g(p) \big)^{-1} \le \bigg( \frac {\log z}{\log w} \bigg)^{\kappa} \bigg( 1 + \frac K{\log w} \bigg)
\end{equation}
whenever $2 \le w < z$. The next lemma is a version of the upper-bound Rosser--Iwaniec sieve: see Iwaniec \cite[Theorem 1]{Iwan80a}.

\begin{lemma}\label{lB}
  Let $z \ge 2$, $s \ge 1$, and let $\mathcal A$ be an integer sequence. Suppose that $N$, the arithmetic function $g$ and the remainders $r(d)$ are defined by \eqref{2.9}, and that \eqref{2.10} holds for some absolute constants $\kappa \ge 0$ and $K \ge 2$. Then
  \[
    \sum_{a \in \mathcal A} \Phi(a, z) \le NV(z) \big( 1 + O \big( e^{-s} \big) \big)  + \sum_{d \le z^s} \mu(d)^2 |r(d)|,
  \]
  where $V(z) = \prod_{p \le z} \big( 1 - g(p) \big)$. The implied constant depends at most on $\kappa$ and $K$.
\end{lemma}

\subsection{The singular series}
\label{s2.3}

In this section, we collect the necessary information about the singular series for sums of a prime and a square of a prime and for sums of three squares of primes. Let $S(q, a)$ be given by~\eqref{0.1}. We define
\begin{gather}
  A_2(n, q) = \frac {\mu(q)}{\phi(q)^2} \sum_{ \substack{ 1 \le a \le q\\ (a, q) = 1}} S(q, a)e_q(-an), \quad A_3(n, q) = \frac 1{\phi(q)^3} \sum_{ \substack{ 1 \le a \le q\\ (a, q) = 1}} S(q, a)^3e_q(-an), \notag\\
  \mathfrak S_j(n, P) = \sum_{q \le P} A_j(n, q), \quad \mathfrak P_j(n, P) = \prod_{p \le P} \big( 1 + A_j(n, p) + A_j(n, p^2) + \cdots \big). \label{2.11}
\end{gather}
Note that $\mathfrak S_2(n, P)$ is the sum defined earlier in \eqref{0.8}. These sums and products were studied in great detail by Schwarz \cite[\S\S 2--3]{Schw61}. Here is a list of some facts that can be found there:
\begin{itemize}
  \item [i)] $A_j(n, q)$ is multiplicative in $q$.
  \item [ii)] $A_3(n, p^k) = 0$ when $p \ge 3, k \ge 2$ or $p = 2, k \ge 4$.
  \item [iii)] If $n \in \mathcal H_2$, then $A_2(n, p) > -1$ for all $p$.
  \item [iv)] If $n \in \mathcal H_3$, then $A_3(n, 2^j) \ge 0$ and $A_3(n, p) > -1$ for all $p \ge 3$.
  \item [v)] $\sum\limits_{n = 1}^{q_1q_2} A_j(n, q_1)A_j(n, q_2) = 0$ when $q_1 \ne q_2$.
\end{itemize}
Furthermore, it is not difficult to show that
\begin{equation}\label{2.11a}
  A_2(n, p) = \begin{cases}
    (p - 1)^{-1} & \text{if } p \mid n, \\
    \big( \frac np \big)p^{-1} + O \big( p^{-2} \big) & \text{if } p \nmid n,
  \end{cases}
\end{equation}
where $\big( \frac np \big)$ is the Legendre symbol modulo $p$. There is also a similar expression for $A_3(n, p)$ (see Mikawa \cite[(4.1)]{Mika97}), from which we can deduce that $|A_3(n, p)| \le 3p^{-1} + O\big( p^{-2} \big)$. Hence,
\begin{equation}\label{2.12}
  |A_j(n, q)| \ll q^{-1} \prod_{p \mid q} \big( 1 + p^{-1} \big)^c \ll q^{-1}(\log\log q)^c.
\end{equation}
We also have
\begin{equation}\label{e3.8}
  (\log P)^{-1} \ll \mathfrak P_2(n, P) \ll \log P, \quad (\log P)^{-3} \ll \mathfrak P_3(n, P) \ll (\log P)^3.
\end{equation}
Finally, we state and prove a lemma, which allows us to approximate $\mathfrak S_j(n, P)$ by $\mathfrak P_j(n, P)$ on average over $n$, provided that $P$ is small compared to $n$. The lemma is essentially a generalization of a result of Schwarz \cite[Satz 1]{Schw61}, but our proof is considerably shorter.

\begin{lemma}\label{l8}
  Let $A \ge 2$ and $\eps > 0$ be fixed. Suppose that $x^{\eps} \le y \le x$ and $P \le Q \le \exp\big( (\log x)^{1 - \eps} \big)$. Then
  \begin{equation}\label{e3.9}
    \sum_{x < n \le x + y} \big( \mathfrak S_j(n, P) - \mathfrak P_j(n, Q) \big)^2 \ll yP^{-1}\log x + y(\log x)^{-A}.
  \end{equation}
\end{lemma}

\begin{proof}
  We may assume that $\eps < \frac 14$. Put $Q_2 = \prod_{p \le Q} p$, $Q_3 = 4Q_2$, and $Q_0 = y^{1/3}$. We have
  \begin{equation}\label{2.15}
    \mathfrak P_j(n, Q) = \sum_{q \le Q_j} A_j(n, q)\Psi(q, Q) = \sum_{q \le Q_0} A_j(n, q)\Psi(q, Q) + \Sigma,
  \end{equation}
  where
  \[
    \Sigma = \sum_{Q_0 < q \le Q_j} A_j(n, q)\Psi(q, Q) \ll  \sum_{Q_0 < q \le Q_j} q^{-1}(\log\log q)^c\Psi(q, Q).
  \]
  An appeal to Lemma \ref{l4} then yields
  \[
    \Sigma \ll (\log\log Q_j)^c \sum_{Q_0 < q \le Q_j} q^{-1}\Psi(q, Q) \ll (\log Q)^c \exp \big( -(\log Q_0)/(2\log Q) \big).
  \]
  Since $\Psi(q, Q) = 1$ when $1 \le q \le P$, we deduce from this inequality and \eqref{2.15} that
  \[
    \mathfrak S_j(n, P) - \mathfrak P_j(n, Q) = \sum_{P < q \le Q_0} \theta_q A_j(n, q) + O \big( (\log x)^{-A - 2} \big),
  \]
  where $\theta_q = 1 - \Psi(q, Q)$. Since the sum over $q$ does not exceed $(\log x)^2$ (recall \eqref{2.12}), the desired conclusion then follows from the bound
  \begin{equation}\label{2.16}
    \sum_{x < n \le x + y} \sum_{P < q_1, q_2 \le Q_0} \theta_{q_1}\theta_{q_2} A_j(n, q_1)A_j(n, q_2) \ll yP^{-1}\log x + y(\log x)^{-A}.
  \end{equation}
  By \eqref{2.12},
  \[
    \sum_{x < n \le x + y} \sum_{P < q \le Q_0} \theta_q^2 A_j(n, q)^2 \ll \sum_{x < n \le x + y} \sum_{q > P} q^{-2}(\log\log q)^c \ll yP^{-1}\log x.
  \]
  On the other hand, when $q_1 \ne q_2$, \eqref{2.12} and v) above yield
  \[
    \bigg| \sum_{x < n \le x + y} A_j(n, q_1)A_j(n, q_2) \bigg| \le 2 \sum_{n = 1}^{q_1q_2} |A_j(n, q_1)A_j(n, q_2)| \ll (q_1q_2)^{\eps/2},
  \]
  whence
  \[
    \sum_{x < n \le x + y} \sum_{ \substack{P < q_1 < q_2 \le Q_0\\ q_1 \ne q_2}} \theta_{q_1}\theta_{q_2} A_j(n, q_1)A_j(n, q_2)
    \ll Q_0^{2 + \eps} \ll y^{3/4}.
  \]
  This establishes \eqref{2.16}.
\end{proof}

\section{Proof of Theorem \ref{thm1}}
\label{s4}

We first verify the hypotheses of Proposition \ref{p1} for $R(n; \varpi, \lambda_0)$, where $\varpi$ is the indicator function of the primes and $\lambda_0$ is defined by \eqref{e2.22}. These functions clearly satisfy hypotheses (A$_{j.1}$) in \S\ref{s1}. When $\theta_1 > \frac 7{12}$, $\varpi$ satisfies hypothesis (A$_{1.2}$) with $f_1(u) = (\log u)^{-1}$ ($u \ge 2$). This is a short interval form of the Siegel--Walfisz theorem that can be established by the same methods as Huxley's theorem on primes in short intervals. The same methods establish also hypothesis (A$_{2.2}$) for $\lambda_0$ with
\[
  f_2(u) = \frac 1{\log u} + \int_{z_0}^{\sqrt u} \frac {dt}{t(\log t)(\log(u/t))} \qquad (u \in \mathbf I_2),
\]
provided that $\theta_2 > \frac 7{12}$. The first term in the above sum accounts for the primes in the support of $\lambda_0$, and the second term accounts for products $p_1p_2$ with $z_0 < p_1 \le p_2$. (The reader can find a justification of hypothesis (A$_{2.2}$) in the case when $\lambda_2 = \varpi$ in \cite[Lemma 5.1]{LiZh97} or in Mikawa and Peneva \cite[Lemma 2]{MiPe07}.)

Finally, we consider hypothesis (A$_{2.3}$). We set $z_1 = Y^{1/6 + \eps/2}$ and note that
\[
  \lambda_0(m) = \Phi(m, z_0) = \Phi(m, z_1) - \sum_{ \substack{ z_1 < p \le z_0\\ p \mid m}} \Phi \big( mp^{-1}, p \big).
\]
Since $mp^{-1} \le Y^{1/3 - \eps/2}$ in the sum above, we have $\Phi\big( mp^{-1}, p \big) = \varpi \big( mp^{-1} \big) = \Phi\big( mp^{-1}, z_1 \big)$. Hence,
\begin{equation}\label{e4.1}
  \lambda_0(m) = \Phi(m, z_1) - \sum_{ \substack{ z_1 < p \le z_0\\ p \mid m}} \Phi \big( mp^{-1}, z_1 \big) = \lambda_0'(m) - \lambda_0''(m), \quad \text{say}.
\end{equation}
Suppose that $\alpha \in \mathfrak m$ and $\theta_2 \ge \frac 23 + \eps$ (note that the latter condition ensures that $z_1 \le HY^{-1/2}L^{-2B}$). Then Lemma \ref{l3} with $x = Y$, $y = H$, $(m, k) = (1, m)$ and $z = z_1$ establishes hypothesis (A$_{2.3}$) for $\lambda_0'$; the same lemma with $x = Y$, $y = H$, $(m, k) = (p, mp^{-1})$ and $z = z_1$ establishes hypothesis (A$_{2.3}$) for $\lambda_0''$. The hypothesis (A$_{2.3}$) for $\lambda_0$ then follows from \eqref{e4.1}. We remark that the choice of $z_0$ in \eqref{e2.22} is determined by the hypothesis on $M$ in the application of Lemma \ref{l3} to $\lambda_0''$.

Suppose now that $\theta_1 \ge \frac 7{12} + \frac 12\eps$ and $\theta_2 \ge \frac 23 + \eps$. Having verified all the hypotheses of Proposition \ref{p1}, we can then apply that proposition to get
\begin{equation}\label{e4.2}
  R(n; \varpi, \lambda_0) = \mathfrak S_2(n, P)\mathfrak I(n; \varpi, \lambda_0) + O \big( Y^{1/2}L^{-A} \big)
\end{equation}
for almost all $n \in \mathcal H_2 \cap (X, X + H]$. Note that with the above choices of $f_1$ and $f_2$, we have
\begin{equation}\label{e4.3}
  \mathfrak I(n; \varpi, \lambda_0) = r_2(n)\big( 1 + O \big( L^{-1}\log L \big) \big).
\end{equation}
Combining \eqref{e2.23}, \eqref{e3.8}, \eqref{e3.9}, \eqref{e4.2} and \eqref{e4.3}, we obtain the asymptotic formula
\begin{equation}\label{e4.4}
  R_2(n) = \mathfrak S_2(n, P)r_2(n)\big( 1 + O \big( L^{-1}\log L \big) \big) - R_0(n)
\end{equation}
for almost all $n \in \mathcal H_2 \cap (X, X + H]$, provided that $X^{7/18 + \eps} \le H \le X^{1 - \eps}$ and $P = (\log X)^B$ with $B$ sufficiently large in terms of $A$. Therefore, Theorem \ref{thm1} follows from the following proposition.

\begin{proposition}\label{p2}
  Let $A > 0$, $\delta > 0$ and $\eps > 0$ be fixed, and suppose that $Y^{\eps} \le H \le Y^{1 - \eps}$. There exists a $B_0 = B_0(A) > 0$ such that when $B \ge B_0$, one has
  \[
    R_0(n) \ll r_2(n)\mathfrak S_2(n, P)L^{-1 + \delta}
  \]
  for all but $O \big( HL^{-A} \big)$ integers $n \in \mathcal H_2 \cap (X, X + H]$.
\end{proposition}

\begin{proof}
  We estimate $R_0(n)$ by means of an upper-bound sieve. Observe that $R_0(n)$ is the number of primes in the sequence
  \[
    \mathcal A = \big\{ m \in I_1 \; \big| \; m = n - (p_1p_2)^2 \text{ with } \; p_1 \in \mathbf I_3, \; p_1 \le p_2, \; p_1p_2 \in \mathbf I_2 \big\},
  \]
  where $\mathbf I_3 = \big[ z_0, Y^{1/4} \big)$. Hence,
  \begin{equation}\label{e4.5}
    R_0(n) \le \sum_{m \in \mathcal A} \Phi(m, z),
  \end{equation}
  where $z$ is any parameter with $2 \le z \le X^{1/2}$. We now proceed to apply Lemma \ref{lB} to the right side of \eqref{e4.5}.

  When $X < n \le X + H$, $|\mathcal A|$ is the number of products $p_1p_2$, where
  \[
    p_1 \in \mathbf I_3, \quad p_1p_2 \in \mathbf I_2, \quad p_1 \le p_2.
  \]
  Thus, upon writing $\mathbf J(p)$ for the interval defined by the conditions $px \in \mathbf I_2$ and $x \ge p$, we deduce from the Prime Number Theorem that
  \[
    |\mathcal A| = N + O \big( Y^{1/2}\exp\big( -L^{1/2} \big) \big),
  \]
  where
  \begin{equation}\label{e4.6}
    N = \sum_{p \in \mathbf I_3} \int_{\mathbf J(p)} \frac {du}{\log u} \ll Y^{1/2}L^{-2}\log L.
  \end{equation}
  Suppose that $d$ is a squarefree integer, with $d \le Y^{1/8}$. Then
  \[
    |\mathcal A_d| = \sum_{ h \in \mathcal R_d} \sum_{p_1 \in \mathbf I_3} \sum_{ \substack{ p_2 \in \mathbf J(p_1)\\ p_1p_2 \equiv h \pmodulo d}} 1,
  \]
  where $\mathcal R_d$ represents a maximal set of incongruent solutions of $x^2 \equiv n \pmodulo d$. In particular, we have $|\mathcal A_d| = 0$ when $(d, n) > 1$. We now define
  \[
    r(d) = |\mathcal A_d| - g(d)N, \qquad
    g(d) = \begin{cases}
      \phi(d)^{-1}|\mathcal R_d| & \text{if } (n, d) = 1, \\
      0                          & \text{if } (n, d) > 1,
    \end{cases}
  \]
  and note that
  \[
    |\mathcal R_d| = \prod_{p \mid d} \bigg( 1 + \bigg( \frac np \bigg) \bigg),
  \]
  $\big( \frac np \big)$ being the Legendre symbol modulo $p$. We note that when $n \in \mathcal H_2$, $g$ satisfies the hypothesis \eqref{2.10} of Lemma \ref{lB} with $\kappa = 2$. Furthermore, it follows from the above definitions that if $D \le Y^{1/8}$, we have
  \[
    \sum_{d \le D} \mu(d)^2 |r(d)| \le \sum_{p \in \mathbf I_3} \sum_{d \le D} \tau(d) \max_{(a, d) = 1} \max_{x \in \mathbf J(p)} \bigg| \pi(x; d, a) - \frac 1{\phi(d)} \int_2^x \frac {dt}{\log t} \bigg|,
  \]
  where $\pi(x; d, a)$ is the number of primes $p \equiv a \pmodulo d$ with $p \le x$. The sum over $d$ can be estimated by means of the Bombieri--Vinogradov theorem and Cauchy's inequality. Thus, for any fixed $A > 0$ and any $D \le Y^{1/8}L^{-B(A)}$, we obtain the bound
  \[
    \sum_{d \le D} \mu(d)^2 |r(d)| \ll \sum_{p \in I_3} Y^{1/2}p^{-1}L^{-A} \ll Y^{1/2}L^{-A}.
  \]

  We now apply Lemma \ref{lB} with $D = Y^{1/9}$ and $z = \exp\big( L^{1 - \delta/2} \big)$ to the sequence $\mathcal A$. We get
  \begin{equation}\label{e4.7}
    \sum_{m \in \mathcal A} \Phi(m, z) \ll N \prod_{p \le z} \big( 1 - g(p) \big) + Y^{1/2}L^{-A},
  \end{equation}
  where $A > 0$ can be taken arbitrarily large. Comparing the definition of $g$ and \eqref{2.11a}, we find that when $n \in \mathcal H_2$,
  \begin{align}\label{e4.8}
    \prod_{p \le z} \big( 1 - g(p) \big) \ll \prod_{p \le z} \bigg( 1 - \frac 1p \bigg) \cdot
    \prod_{p \le z} \big( 1 + A_2(n, p) \big) \ll \frac {\mathfrak P_2(n, z)}{\log z}.
  \end{align}
  Here, $\mathfrak P_2(n, z)$ is the partial singular product defined in \eqref{2.11}. Combining the lower bound \eqref{e3.8} and inequalities \eqref{e4.5}--\eqref{e4.8}, we conclude that
  \[
    R_0(n) \ll Y^{1/2}L^{-3 + \delta}\mathfrak P_2(n, z).
  \]
  Finally, by \eqref{2.12} and Lemma \ref{l8} with $x = X$, $y = H$ and $Q = z$, the asymptotic formula
  \[
    \mathfrak P_2(n, z) = \mathfrak S_2(n, P)\big( 1 + O \big( L^{-1} \big) \big)
  \]
  holds for almost all integers $n \in \mathcal H_2 \cap (X, X + H]$, provided that $P \ge L^{A + 5}$.
\end{proof}

\section{Proof of Theorem \ref{thm2}}
\label{s5}

As we stated already in \S\ref{s1}, the proof of Theorem \ref{thm2} makes use of two pairs of functions, $\lambda_1^{\pm}$ and $\lambda_2^{\pm}$ satisfying \eqref{e2.25} and \eqref{e2.26}, respectively. We borrow the functions $\lambda_1^{\pm}$ from Baker, Harman and Pintz \cite{BaHaPi97}: we choose $\lambda_1^-(m) = a_0(m)$ and $\lambda_1^+(m) = a_1(m)$, where $a_0$ and $a_1$ are the functions constructed in \cite{BaHaPi97} (see \cite[\S4]{BaHaPi97} for details). We note that, by construction, these functions satisfy hypotheses (A$_{1.1}$) and (A$_{1.2}$) of Proposition~\ref{p1} when $\theta_1 \ge 0.55 + \eps$.

Next, we turn to the construction of $\lambda_2^{\pm}$. As hypothesis (A$_{2.3}$) is the most demanding among the requirements imposed on $\lambda_2$ in Proposition \ref{p1}, our construction focuses on satisfying that hypothesis. Let
\begin{equation}\label{e5.1}
  U = Y^{\eps/2}, \quad V = HY^{-1/2 - \eps/2}, \quad W = Y^{1/2}V^{-1}.
\end{equation}
Recall also the definition of $z_0$ in \eqref{e2.22}. We apply twice Buchstab's identity
\begin{equation}\label{e5.2}
  \Phi(m, z) = \Phi(m, w) - \sum_{\substack{ w < p \le z\\ p \mid m}} \Phi\big( mp^{-1}, p \big) \qquad (2 \le w < z)
\end{equation}
to decompose $\lambda_0$ as follows:
\begin{align}\label{e5.3}
  \lambda_0(m) &= \Phi(m, V) - \sum_{\substack{ V < p \le z_0\\ p \mid m}} \Phi\big( mp^{-1}, V \big) + \sum_{\substack{ V < p_2 < p_1 \le z(p_2) \\ p_1p_2 \mid m}} \Phi\big( m(p_1p_2)^{-1}, p_2 \big) \notag\\
  &= \gamma_1(m) - \gamma_2(m) + \gamma_3(m), \quad \text{say}. 
\end{align}
Here, we have $z(p) = \min\big( z_0, Y^{1/2}p^{-2} \big)$. In particular, when $\theta_2 \ge \frac 23 + \eps$, the sum $\gamma_3$ is empty and \eqref{e5.3} turns into \eqref{e4.1}. We now split $\gamma_3(m)$ into two subsums. We have
\begin{equation}\label{e5.4}
  \gamma_3(m) = \bigg\{ \sum_{\substack{ \cdots\\ p_1p_2 < W}} + \sum_{\substack{ \cdots\\ p_1p_2 \ge W}} \bigg\} \Phi\big( m(p_1p_2)^{-1}, p_2 \big) = \gamma_4(m) + \gamma_5(m), \quad \text{say},
\end{equation}
where the $\cdots$ represent the summation conditions $V < p_2 < p_1 \le z(p_2)$ and $p_1p_2 \mid m$. We are now in position to define $\lambda_2^-$. We set
\begin{equation}\label{e5.5}
  \lambda_2^-(m) = \gamma_1(m) - \gamma_2(m) + \gamma_5(m).
\end{equation}
Note that, by \eqref{e5.3} and \eqref{e5.4}, we have $\lambda_2^-(m) = \lambda_0(m) - \gamma_4(m)$, so $\lambda_2^-$ satisfies \eqref{e2.26}. Furthermore, by virtue of \eqref{e2.22} and \eqref{e5.1}, we can use Lemma \ref{l3} to verify hypothesis (A$_{2.3}$) for $\gamma_1$ and $\gamma_2$. Finally, in $\gamma_5$, we have $U \le m(p_1p_2)^{-1} \le V$, so we can apply Lemma \ref{l2} with $(m, k) = \big( m(p_1p_2)^{-1}, p_1p_2 \big)$ to verify hypothesis (A$_{2.3}$) for $\gamma_5$. We conclude that $\lambda_2^-$ satisfies both \eqref{e2.26} and hypotheses (A$_{2.1}$) and (A$_{2.3}$) of Proposition~\ref{p1}. When $\theta_2 > \frac 7{12}$, $\lambda_2^-$ satisfies also hypothesis (A$_{2.2}$), though this may require some explanation.

As we mentioned earlier, hypothesis (A$_{2.2}$) holds for $\lambda_2 = \varpi$ when $\theta_2 > \frac 7{12}$. One way to prove this is to use \eqref{e5.2} to decompose $\varpi$ into a linear combination of functions similar to our $\gamma_i$'s and then to establish hypothesis (A$_{2.2}$) for each function in that decomposition. Applying that same decomposition to $\lambda_2^-$ instead to $\varpi$ is equivalent to taking the intersection of two partitions of a set. Therefore, such a decomposition of $\lambda_2^-$ will produce more terms than the respective decomposition of $\varpi$, but every such term will be a subsum of a sum appearing in the decomposition of $\varpi$. Thus, the same results, which establish (A$_{2.2}$) for all terms in the decomposition of $\varpi$, will establish (A$_{2.2}$) for all terms in the decomposition of $\lambda_2^-$.

We now proceed with the construction of $\lambda_2^+$. By \eqref{e5.2},
\begin{align}\label{e5.6}
  \lambda_0(m) &= \Phi(m, V) - \sum_{\substack{ V < p \le z_1\\ p \mid m}} \Phi\big( mp^{-1}, p \big) - \sum_{\substack{ z_1 < p \le z_0\\ p \mid m}} \Phi\big( mp^{-1}, p \big) \notag\\
  &= \beta_1(m) - \beta_2(m) - \beta_3(m), \quad \text{say}. 
\end{align}
Here, $z_1 = \max \big( V, z_0^{1/2} \big)$. Note that when $\theta_2 \ge \frac 58$, $z_1 = V$ and the sum $\beta_2$ is empty. Suppose now that $\theta_2 < \frac 58$ (and hence, $z_1 = z_0^{1/2}$). We apply \eqref{e5.2} two more times to $\beta_2$:
\begin{align}\label{e5.7}
  \beta_2(m) &= \sum_{\substack{ V < p \le z_1\\ p \mid m}} \Phi\big( mp^{-1}, V \big) - \sum_{\substack{ V < p_2 < p_1 \le z_1\\ p_1p_2 \mid m}} \Phi\big( m(p_1p_2)^{-1}, V \big) \notag\\
  &\qquad \qquad + \sum_{\substack{ V < p_3 < p_2 < p_1 \le z_1 \\ p_1p_2p_3 \mid m}} \Phi\big( m(p_1p_2p_3)^{-1}, p_3 \big) \notag \\
  &= \beta_4(m) - \beta_5(m) + \beta_6(m), \quad \text{say}. 
\end{align}
We define
\begin{equation}\label{e5.8}
  \lambda_2^+(m) = \beta_1(m) - \beta_4(m) + \beta_5(m).
\end{equation}
By \eqref{e5.6} and \eqref{e5.7}, we have $\lambda_2^+(m) = \lambda_0(m) + \beta_3(m) + \beta_6(m)$, so $\lambda_2^+$ satisfies \eqref{e2.26} and hypothesis (A$_{2.1}$) of Proposition~\ref{p1}. Hypothesis (A$_{2.2}$) holds when $\theta_2 > \frac 7{12}$ for the same reasons as in the case of $\lambda_2^-$. Finally, $\lambda_2^+$ satisfies hypothesis (A$_{2.3}$), because Lemma \ref{l3} can be applied to each of the three terms on the right side of~\eqref{e5.8}.

Suppose now that $\lambda_i^{\pm}$ are the above functions and that $\theta_1 \ge 0.55 + \eps$ and $\frac 7{12} < \theta_2 \le \frac 23$. With these choices, we can apply Proposition \ref{p1} to each of the three terms on the right side of \eqref{e2.27}. We deduce that
\begin{equation}\label{e5.9}
  R(n; \varpi, \lambda_0) \ge \mathfrak S_2(n, P)\mathfrak I(n)(1 + o(1))
\end{equation}
for almost all $n \in \mathcal H_2 \cap (X, X + H]$. Here,
\[
  \mathfrak I(n) = \sum_{ \substack{ m_1 + m_2^2 = n\\ m_j \in \mathbf I_j}} \big( f_1^+(m_1)f_2^-(m_2) + f_1^-(m_1)f_2^+(m_2) - f_1^+(m_1)f_2^+(m_2) \big),
\]
$f_j^{\pm}$ being the smooth functions appearing in hypotheses (A$_{j.2}$).

The functions $f_j^{\pm}$ arise via applications of Lemma \ref{l5}. For example, when $\theta_2 > \frac 58$, we have $\lambda_2^+(m) = \Phi(m, V)$, and Lemma \ref{l5} gives
\[
  \sum_{m \le x} \lambda_2^+(m) = \frac 1{\log V} \sum_{z < m \le x} w \bigg( \frac {\log m}{\log V} \bigg) + O \big( Y^{1/2}L^{-A} \big)
\]
for any fixed $A > 0$ and any $x \le Y^{1/2}$. Hence, in this case, we have
\[
  f_2^+(m) = \frac 1{\log V} w \bigg( \frac {\log m}{\log V} \bigg) \qquad (m \ge V).
\]
Furthermore, the functions $f_j^{\pm}$ satisfy asymptotic formulas of the form
\begin{equation}\label{e5.10}
  \sum_{ \substack{m \in \mathbf I_j\\ m \le x}} f_j^{\pm}(m) = \big( \sigma_j^{\pm} + O\big( L^{-1} \big) \big) \sum_{ \substack{m \in \mathbf I_j\\ m \le x}} \frac 1{\log m},
\end{equation}
where $\sigma_j^{\pm} = \sigma_j^{\pm}( \theta_j )$ are numbers depending only on $\theta_1$ and $\theta_2$. The values of $\sigma_1^{\pm}$ are estimated in \cite{BaHaPi97}: when $\theta_1 \ge 0.55 + \eps$, we have
\begin{equation}\label{e5.11}
  \sigma_1^+ < 1.01, \quad \sigma_1^- > 0.99.
\end{equation}
On the other hand, the values of $\sigma_2^{\pm}$ arising from the above construction of $\lambda_2^{\pm}$ are
\begin{align*}
  \sigma_2^- &= 1 - \iint_{\mathcal D_2^-} w \bigg( \frac {1 - u_1 - u_2}{u_2} \bigg) \, \frac {du_1du_2}{u_1u_2^2} + O(\eps),\\
  \sigma_2^+ &= 1 + \int_{1/4}^{1/2} w \bigg( \frac {1 - u}{u} \bigg) \, \frac {du}{u^2} + \iiint_{\mathcal D_2^+} w \bigg( \frac {1 - u_1 - u_2 - u_3}{u_3} \bigg) \, \frac {du_1du_2du_3}{u_1u_2u_3^2} + O(\eps),\\
\end{align*}
where
\begin{align*}
  \mathcal D_2^-:& \quad 2\theta_2 - 1 < u_2 < u_1 < \sfrac 12, \; u_1 + 2u_2 < 1, \; u_1 + u_2 < 2 - 2\theta_2, \\
  \mathcal D_2^+:& \quad 2\theta_2 - 1 < u_3 < u_2 < u_1 < \sfrac 14.
\end{align*}
A computer calculation then yields
\begin{equation}\label{e5.12}
  \sigma_2^-(\sfrac 35) > 0.22, \quad \sigma_2^+(\sfrac 35) < 2.26.
\end{equation}
Combining \eqref{e5.10}--\eqref{e5.12}, we get
\[
  \mathfrak I(n) \ge r_2(n) \big( 0.17 + O \big( L^{-1} \big) \big)
\]
Inserting this bound into \eqref{e5.9}, we obtain
\begin{equation}\label{e5.13}
  R(n; \varpi, \lambda_0) \ge \mathfrak S_2(n, P)r_2(n)(0.17 + o(1))
\end{equation}
for almost all $n \in \mathcal H_2 \cap (X, X + H]$.

Finally, we choose $\theta_1 = 0.55 + \eps$ and $\theta_2 = \frac 35 - 2\eps$. Theorem \ref{thm2} is a direct consequence of \eqref{e2.23}, \eqref{e5.13} and Proposition \ref{p2}. \qed

\section{Sums of three and four squares}
\label{s6}

\subsection{Proof of Theorem \ref{thm3}}
\label{s6.1}

The argument is similar to the proof of Theorem \ref{thm2}, so we only outline the differences between the two proofs. Let $R_3(n)$ denote the number of representations of $n$ in the form
\[
  R_3(n) = \sum_{ \substack{ p_1^2 + p_2^2 + p_3^2= n\\ p_1^2 + p_2^2 \in \mathbf I_1, p_3 \in \mathbf I_2\\ }} 1.
\]
In place of the quantity defined in \eqref{e2.1}, we use
\begin{equation}\label{e6.1}
  R(n; \lambda_1, \lambda_2) = \sum_{ \substack{ m_1^2 + m_2^2 + m_3^2= n\\ m_1^2 + m_2^2 \in \mathbf I_1, m_3 \in \mathbf I_2,\\ }} \!\!\! \lambda_1(m_1, m_2)\lambda_2(m_3).
\end{equation}
We set $\lambda_1(m, k) = \varpi(m)\varpi(k)$ and $\lambda_2(m) = \lambda_2^-(m)$, where $\lambda_2^-$ is the function defined in \eqref{e5.5}. Similarly to \eqref{e2.23} and \eqref{e2.27}, we have
\begin{equation}\label{e6.2}
  R_3(n) \ge R(n; \lambda_1, \lambda_2) - R_0(n),
\end{equation}
where $R_0(n)$ is the number of solutions of the equation
\begin{equation*}
  p_1^2 + p_2^2 + (p_3p_4)^2 = n
\end{equation*}
in primes $p_1, \dots, p_4$ subject to
\begin{equation*}
  p_1^2 + p_2^2 \in \mathbf I_1, \quad z_0 < p_3 \le Y^{1/4}, \quad p_3 \le p_4, \quad p_3p_4 \in \mathbf I_2.
\end{equation*}

Suppose again that $A > 0$ is a fixed (large) real and set
\begin{equation}\label{e6.3}
  P = L^B, \quad Q_0 = YP^{-3}, \quad Q = HP^{-1},
\end{equation}
where $B$ is a parameter to be chosen later in terms of $A$. Similarly to Proposition \ref{p2}, one can show that
\begin{equation}\label{e6.4}
  R_0(n) \ll \mathfrak S_3(n, P)Y^{1/2}L^{-4+\delta}
\end{equation}
for almost all $n \in \mathcal H_3 \cap (X, X + H]$. Here, $\mathfrak S_3(n, P)$ is defined by \eqref{2.11}.

Next, we use the circle method to evaluate the quantity $R(n; \lambda_1, \lambda_2)$ in \eqref{e6.2}. The orthogonality relation \eqref{e2.2} holds with $S_1(\alpha)$ replaced by the sum
\[
  S_1(\alpha) = \sum_{p_1^2 + p_2^2 \in \mathbf I_1} e\big( \alpha (p_1^2 + p_2^2) \big).
\]
We define the sets of major and minor arcs as before. By the discussion in \S\ref{s5}, $\lambda_2$ satisfies hypotheses (A$_{2.j}$) in \S\ref{s1}. Since
\[
  I_1 = \int_0^1 |S_1|^2 \, d\alpha = \sum_{m \in \mathbf I_1} \bigg( \sum_{p_1^2 + p_2^2 = m} 1 \bigg)^2 \ll YL^3,
\]
we obtain similarly to \eqref{e2.20} that
\begin{equation}\label{e6.5}
  \sum_{X < n \le X + H} \bigg| \int_{\mathfrak m} S_1(\alpha)S_2(\alpha)e(-\alpha n) \, d\alpha \bigg|^2 \ll HYL^{-A}.
\end{equation}
Furthermore, similarly to \eqref{e2.9} and \eqref{e2.12}, we have
\begin{equation}\label{e6.6}
  \int_{\mathfrak M} |S_1(S_2 - S_2^*)| \, d\alpha \ll Y^{1/2}P^{-1/2 + \eta}
\end{equation}
and (recall \eqref{e2.11a})
\begin{equation}\label{e6.7}
  \int_{\mathfrak m_0} |S_1S_2^*| \, d\alpha \ll Y^{1/2}P^{-1/2 + \eta}.
\end{equation}
Define
\[
  T_1(\beta) = \sum_{m \in \mathbf I_1} f_1(m)e(\beta m), \quad f_1(m) = \int_0^1 \frac {du}{\sqrt{u(1 - u)}(\log mu)(\log m(1 - u))}.
\]
When $\alpha \in \mathfrak M(q, a) \cap \mathfrak M_0$ and $\theta_1 > \frac 7{12}$, a variant of Mikawa and Peneva \cite[Lemma 3]{MiPe07} yields
\[
  | S_1(\alpha) - S_1^*(\alpha)| \ll q^{\eta} YP^{-10}(1 + Y|\alpha - a/q|),
\]
where
\[
  S_1^*(\alpha) = \frac {\pi}4 \frac {S(q, a)^2}{\phi(q)^2} T_1(\alpha - a/q).
\]
Hence,
\begin{equation}\label{e6.8}
  \int_{\mathfrak M_0} |(S_1 - S_1^*)S_2^*| \, d\alpha \ll Y^{1/2}P^{-1 + \eta}.
\end{equation}
Finally, we have
\begin{equation}\label{e6.9}
  \int_{\mathfrak M_0} S_1^*(\alpha)S_2^*(\alpha)e(-\alpha n) \, d\alpha = \frac {\pi}4\mathfrak S_3(n, P)\mathfrak I(n; \lambda_2) + O \big(Y^{1/2}P^{-1} \big),
\end{equation}
where $\mathfrak S_3(n, P)$ is defined in \eqref{2.11} and
\[
  \mathfrak I(n; \lambda_2) = \int_{-1/2}^{1/2} T_1(\beta)T_2(\beta)e(-\beta n) \, d\beta = \sum_{ \substack{ m_1 + m_2^2 = n\\ m_i \in \mathbf I_i}} f_1(m_1)f_2^-(m_2).
\]
Combining \eqref{e6.6}--\eqref{e6.9}, we conclude that
\begin{equation}\label{e6.10}
  \int_{\mathfrak M} S_1(\alpha)S_2(\alpha)e(-\alpha n) \, d\alpha = \frac {\pi}4\mathfrak S_3(n, P)\mathfrak I(n; \lambda_2) + O \big(Y^{1/2}P^{-1/2 + \eta} \big).
\end{equation}

From \eqref{e6.2}, \eqref{e6.4}, \eqref{e6.5} and \eqref{e6.10}, we obtain that
\[
  R_3(n) \gg \mathfrak S_3(n, P)Y^{1/2}L^{-3}
\]
for almost all $n \in \mathcal H_3 \cap (X, X + H]$, provided that $\theta_1 > \frac 7{12}$ and the value of $\sigma_2^-(\theta_2)$ in \S\ref{s5} is positive. In particular, upon choosing $\theta_1 = \frac 7{12} + \eps$ and $\theta_2 = \frac 35 - 2\eps$, we deduce Theorem \ref{thm3}. \qed

\subsection{Proof of Corollary \ref{c1}}
\label{s6.2}

Let $E_4'(X)$ be the number of exceptional integers $n$ counted by $E_4(X)$ with $n \not\equiv 1 \pmodulo 5$, and let $E_4''(X) = E_4(X) - E_4'(X)$. By a result of Harman, Watt and Wong \cite[Theorem 3]{HaWaWo04}, there exist prime numbers $q_1$ and $q_2$ such that 
\[
  X^{1/2} - X^{0.2625} < q_j \le X^{1/2} - \sfrac 12X^{0.2625}, \quad q_j \equiv j \pmodulo 5. 
\]
If $n$ is counted by $E_4'(X + H) - E_4'(X)$, then $n - q_1^2$ is counted by $E_3(X_1 + H) - E_3(X_1)$, where $X_1 \asymp X^{0.7625}$. Since $H \ge X_1^{7/20}$, Theorem \ref{thm3} yields
\begin{equation}\label{e6.11}
  E_4'(X + H) - E_4'(X) \le E_3(X_1 + H) - E_3(X_1) \ll HL^{-A},
\end{equation}
for any fixed $A > 0$. Similarly, if $n$ is counted by $E_4''(X + H) - E_4''(X)$, then the integer $n - q_2^2$ is counted by $E_3(X_2 + H) - E_3(X_2)$, where $X_2 \asymp X^{0.7625}$. Hence, Theorem \ref{thm3} yields
\begin{equation}\label{e6.12}
  E_4''(X + H) - E_4''(X) \le E_3(X_2 + H) - E_3(X_2) \ll HL^{-A},
\end{equation}
for any fixed $A > 0$. The result follows from \eqref{e6.11} and \eqref{e6.12}. \qed

\begin{acknowledgement}
  The bulk of this work was completed when the first author visited Shandong University in July of 2007. He would like to use this occasion to express his gratitude to the School of Mathematics for the financial support and the excellent working conditions. The second author is supported by the 973 Program, NSFC Grant \#10531060, and Ministry of Education Grant \#305009.
\end{acknowledgement}

\bibliographystyle{amsplain}

\end{document}